\documentclass{llncs}

\usepackage{color,amsmath,amssymb,proof}

\newcommand{\la}{\lambda}
\newcommand{\de}{\delta}
\newcommand{\e}{e}
\newcommand{\m}{\mathbf} 
\newcommand{\R}{\mathbb{R}}
\newcommand{\Z}{\mathbb{Z}}
\newcommand{\N}{\mathbb{N}}
\newcommand{\iv}[1]{{#1}^{-1}}
\newcommand{\ov}[1]{\overline{#1}}

\newcommand{\De}{\mathrm{\Delta}}
\newcommand{\Ga}{\mathrm{\Gamma}}
\newcommand{\PI}{\mathrm{\Pi}}
\newcommand{\Si}{\mathrm{\Sigma}}
\newcommand{\Om}{\mathrm{\Omega}}
\newcommand{\Th}{\mathrm{\Theta}}

\newcommand{\vty}[1]{\mathcal{#1}}

\newcommand{\lgc}[1]{\mathrm{#1}}

\newcommand{\der}[1]{\vdash_{\lgc{#1}}}
\newcommand{\eq}{\approx}
\newcommand*{\mg}{{\mathcal G}}
\newcommand*{\mh}{{\mathcal H}}
\newcommand*{\h}{\mid}
\newcommand*{\hh}{\!\mid\!}
\newcommand{\starr}{(*)}

\newcommand{\exr}{\textup{\sc (ex)}}

\newcommand{\cutr}{\textup{\sc (cut)}}

\newcommand{\idr}{\textup{\sc (id)}}

\newcommand{\comr}{\textup{\sc (com)}}

\newcommand{\ewr}{\textup{\sc (ew)}}
\newcommand{\mixr}{\textup{\sc (mix)}}
\newcommand{\splitr}{\textup{\sc (split)}}
\newcommand{\cycr}{\textup{\sc (cycle)}}

\newcommand{\emr}{\textup{\sc (em)}}
\newcommand{\gvr}{\textup{\sc (gv)}}
\newcommand{\seq}{\Rightarrow}
\newcommand{\sgr}[1]{\langle #1 \rangle} 
\newcommand{\nsgr}[1]{\langle\langle #1 \rangle\rangle} 

\begin{document}

\title{Proof Theory and Ordered Groups}
\titlerunning{Proof Theory and Ordered Groups}
\author{Almudena Colacito \and George Metcalfe\thanks{Supported by Swiss National Science Foundation 
grant 200021{\_}146748 and the EU Horizon 2020 research and innovation programme under the Marie Sk{\l}odowska-Curie grant agreement No 689176.}}

\institute{
Mathematical Institute, University of Bern, Switzerland\\
\email{\{almudena.colacito,george.metcalfe\}@math.unibe.ch} 
}	

\maketitle


\begin{abstract}
Ordering theorems, characterizing when partial orders of a group extend to total orders, are used to generate hypersequent calculi for varieties of lattice-ordered groups ($\ell$-groups). These calculi are then used to provide new proofs of theorems arising in the theory of ordered groups. More precisely: an analytic calculus for abelian $\ell$-groups is generated using an ordering theorem for abelian groups; a calculus  is generated for $\ell$-groups and  new decidability proofs are obtained for the equational theory of this variety and extending finite subsets of free groups to right orders; and a calculus for representable $\ell$-groups is generated and a new proof is obtained that free groups are orderable.	
\end{abstract}


\section{Introduction}

Considerable success has been enjoyed recently in obtaining uniform algebraic completeness proofs for analytic sequent and hypersequent calculi with respect to varieties of residuated lattices~\cite{Ter07,CGT12,CGT17}. These methods do not encompass, however, ``ordered group-like'' structures: algebras with a group reduct such as lattice-ordered groups ($\ell$-groups)~\cite{AF88,KM94} and others admitting representations via ordered groups such as MV-algebras~\cite{COM99}, GBL-algebras~\cite{JM10}, and varieties of cancellative residuated lattices~\cite{MT10}. Hypersequent calculi have indeed been defined for abelian $\ell$-groups, MV-algebras, and related classes in~\cite{MOG05,MOG08} and for $\ell$-groups in~\cite{GM16}, but the completeness proofs in these papers are largely syntactic, proceeding using cut elimination or restricted quantifier elimination. 

The first aim of the work reported here is to use ordering theorems for groups, characterizing when a partial (right) order of a group extends to a total (right) order, to generate hypersequent calculi for varieties of lattice-ordered groups, thereby taking a first step towards a general algebraic proof theory for ordered group-like structures. A second aim is to then use these calculi to provide new syntactic proofs of various theorems arising in the theory of ordered groups.

More concretely, this paper makes the following contributions:

\begin{itemize}

\item[(i)]	A theorem of Fuchs~\cite{Fuc63} for extending partial orders of abelian groups to total orders is used to generate an analytic (cut-free) hypersequent calculus for the variety of abelian $\ell$-groups. This system can be viewed as a one-sided version of the two-sided hypersequent calculus introduced in~\cite{MOG05}.
\smallskip

\item[(ii)]	A theorem of Kopytov and Medvedev~\cite{KM94} for extending partial right orders of groups to total right orders is used to generate a hypersequent calculus for the variety of $\ell$-groups, a variant of a calculus appearing in~\cite{GM16}. The method also provides a correspondence between validity of equations in $\ell$-groups and the extension of finite subsets of free groups to total right orders, giving new proofs of decidability for these problems. \smallskip

\item[(iii)]	A  theorem of Fuchs~\cite{Fuc63} for extending partial orders of groups to total orders is used to generate a calculus for representable $\ell$-groups (equivalently, ordered groups) and to provide a new proof that free groups are orderable.

\end{itemize}


\section{Ordered Groups}\label{s:orderedgroups}

In this section, we recall some pertinent definitions and basic facts about ordered groups, referring to~\cite{AF88,KM94} for further details. Consider a group $\m{G} = \langle G,\cdot,\iv{},\e \rangle$. A partial order $\le$ of $G$ is called a {\em partial right order} of $\m{G}$ if for all $a,b,c \in G$, 
\[
a \le b \ \Longrightarrow \ ac \le bc.
\]
Its {\em positive cone} $P_\le = \{a \in G : \e < a\}$ is a subsemigroup of $\m{G}$ that omits $\e$. Conversely, if $P$ is a subsemigroup of $\m{G}$ omitting $\e$, then 
\[
a \le^P b \ \Longleftrightarrow \ b\iv{a} \in P \cup \{\e\}
\]
is a partial right order of $\m{G}$ satisfying $P_{\le^P} = P$.  Hence partial right orders of  $\m{G}$ can be identified with subsemigroups of $\m{G}$  omitting $\e$. Note also that for $S \subseteq G$, the  subsemigroup of $\m{G}$ generated by $S$, denoted by $\sgr{S}$, is a partial right order of $\m{G}$ if and only if $\e \not \in \sgr{S}$. Partial left orders of $\m{G}$ are defined analogously. 

A partial left and right order $\le$ of $\m{G}$ is called a {\em partial order} of $\m{G}$. In this case, the positive cone $P_\le$ is a normal subsemigroup of $\m{G}$ omitting $\e$; that is, whenever $a \in P_\le$ and $b \in G$, also $ba\iv{b} \in P_\le$. Conversely, if a subset $P \subseteq G$ has these properties, then $\le^P$ is a partial order of $\m{G}$; hence, partial orders of $\m{G}$ can be identified with normal subsemigroups of $\m{G}$ omitting $\e$. Also, for $S \subseteq G$, the normal subsemigroup of $\m{G}$ generated by $S$, denoted by $\nsgr{S}$, is a partial order of $\m{G}$  if and only if $\e \not \in \nsgr{S}$.

A partial order or partial right order $\le$ of $\m{G}$ is called, respectively, a {\em (total) order} or {\em (total) right order} of $\m{G}$ if $G = P_\le \cup \iv{P_\le} \cup \{\e\}$. Note also that if $\le$ is an order or a right order of $\m{G}$, then the same holds for the {\em inverse order} defined by $a \le^\delta b$ if and only if $b \le a$. In this paper we  focus mostly on (right) orders of a finitely generated free (abelian) group $\m{F}$ and address the following problem.
 
\begin{problem}
Does a given finite $S \subseteq F$ extend to an order or a right order of $\m{F}$?
\end{problem}

We  also consider a purely algebraic perspective on ordered groups.  That is, a {\em lattice-ordered group} (or {\em $\ell$-group}) may be defined as an algebraic structure $\m L = \langle L, \land, \lor, \cdot, \iv{}, \e \rangle$ satisfying
\begin{itemize}
\item[\rm (i)]	$\langle L,\cdot, \iv{}, \e \rangle$ is a group;\smallskip
\item[\rm (ii)]	$\langle L, \land, \lor \rangle$ is a lattice (with $a \le b \, \Leftrightarrow \, a \land b = a$, for all $a,b\in L$); \smallskip
\item[\rm (iii)]	$a \le b \ \Longrightarrow \ cad \le cbd$, for all $a,b,c,d\in L$.
\end{itemize}
It follows also that $\langle L, \land, \lor \rangle$ must be a distributive lattice and that $\m{L}$ satisfies $\e \le a \lor \iv{a}$ for all $a \in L$ (see~\cite{AF88}). If $\le$ is a total order of the group $\langle L,\cdot, \iv{}, \e \rangle$, then $\m{L}$ is called an {\em ordered group} (or {\em o-group}), observing that $\m{L}$ can also be obtained by adding to the group operations the meet and join operations for $\le$. An $\ell$-group whose group operation is commutative is called an {\em abelian $\ell$-group}. 

\begin{example}\label{e:abelian}
Standard examples of abelian $\ell$-groups are  subgroups of the additive group over the real numbers equipped with the usual order, e.g., 
\[
\m{Z} = \langle \Z, \min, \max, +, -, 0 \rangle.
\]
Indeed this algebra generates the variety $\vty{A}$ of all abelian $\ell$-groups~\cite{Wei63}, which means in particular that an equation is valid in $\vty{A}$ if and only if is valid in $\m{Z}$.
\end{example}

\begin{example}
Fundamental  examples of (non-abelian) $\ell$-groups are provided by considering the order-preserving bijections of some totally-ordered set $\langle \Om,\le\rangle$. These form an $\ell$-group $\m{Aut}(\langle \Om, \le \rangle)$ under coordinate-wise lattice operations, functional composition, and functional inverse. Indeed, it has been shown by Holland that every $\ell$-group embeds into an $\ell$-group $\m{Aut}(\langle \Om,\le\rangle)$ for some totally-ordered set $\langle \Om,\le\rangle$~\cite{Hol63}, and that the variety $\vty{LG}$ of $\ell$-groups is generated by $\m{Aut}(\langle \R,\le \rangle)$, where $\le$ is the usual order on $\R$~\cite{Hol76}. This means in particular that an $\ell$-group equation is valid in $\vty{LG}$  if and only if is valid in $\m{Aut}(\langle \R,\le \rangle)$.
\end{example}

Let us turn our attention now to the syntax of $\ell$-groups. We call a variable $x$ and its inverse $\iv{x}$ {\em literals}, and consider terms $s,t,\dots$ built from literals over variables $x_1,x_2,\ldots$, operation symbols $\e$, $\land$, $\lor$, and $\cdot$, defining also inductively
\[
\begin{array}{rclcrclcrcl}
\ov{x} 			& = & \iv{x}  &\qquad 			& \ov{\iv{x}} 	& = & x 		& \qquad 	& \ov{\e}			& = & \e \smallskip \\
\ov{s \land t} 	& = & \ov{s} \lor \ov{t} & 			& \ov{s \lor t} 		& = & \ov{s} \land \ov{t} &      & \ov{s \cdot t}	&  = & \ov{t} \cdot \ov{s}.
\end{array}
\]
Using the strong distributivity properties of the $\ell$-group operations, it follows that every $\ell$-group term is equivalent in $\vty{LG}$ to a term of the form $\land_{i \in I} \lor_{j \in J_i} t_{ij_i}$ where each  $t_{ij_i}$ is a group term. Hence to check the validity of equations in some class $ \vty{K}$ of  $\ell$-groups, it suffices to address the following problem.

\begin{problem}
Given group terms $t_1,\ldots,t_n$, does it hold that
\[
 \vty{K}\, \models\, \e \le t_1 \lor \ldots \lor t_n \ ?
\] 
\end{problem}

\noindent
Let us therefore define a {\em sequent} $\Ga$ as a finite sequence of literals $\ell_1,\ldots,\ell_n$ with inverse $\ov{\Ga}  = \ov{\ell_n},\ldots,\ov{\ell_1}$, and a {\em hypersequent} $\mg$ as a finite set of sequents, written
\[
\Ga_1 \h \ldots \h \Ga_n.
\]
In what follows, we identify a sequent $\ell_1,\ldots,\ell_n$ with the group term $\ell_1 \cdot \ldots \cdot \ell_n$ for $n > 0$ and $\e$ for $n = 0$, and a non-empty hypersequent $\Ga_1 \h \ldots \h \Ga_n$ with the $\ell$-group term $\Ga_1 \lor \ldots \lor \Ga_n$. We will say that a non-empty hypersequent $\mg$ is {\em valid} in a class of $\ell$-groups $\vty{K}$ and write $\vty{K} \models \mg$,  if $\vty{K} \models \e \le \mg$. We will also say that a sequent $\Ga$ is {\em group valid} if $\Ga \eq \e$ is valid in all groups.

A {\em hypersequent rule} is a set of {\em instances}, each instance consisting of a finite set of hypersequents called the {\em premises} and a hypersequent called the {\em conclusion}. Such rules are typically written schematically using $\Ga,\PI,\Si,\De$ and $\mg,\mh$ to denote arbitrary sequents and hypersequents, respectively. A {\em hypersequent calculus} $\lgc{GL}$ is a set of hypersequent rules, and a {\em $\lgc{GL}$-derivation} of a hypersequent $\mg$ is a finite tree of hypersequents with root $\mg$ such that each node and its parents form an instance of a rule of $\lgc{GL}$. In this case, we write $\der{\lgc{GL}} \mg$. A hypersequent rule is said to be $\lgc{GL}$-{\em admissible} if for each of its instances, whenever the premises are $\lgc{GL}$-derivable, the conclusion is  $\lgc{GL}$-derivable.

\begin{remark}
Sequents are often defined (see, e.g.,~\cite{MOG05,MOG08,CGT12,CGT17}) as ordered pairs of finite sequences (or sets or multisets) of terms, and hypersequents as finite multisets of sequents. Here we exploit the strong duality properties of $\ell$-groups to restrict to one-sided sequents and define hypersequents as finite sets of sequents to emphasize the connection with finite sets of group terms. 
\end{remark}


\section{A Hypersequent Calculus for Abelian $\ell$-Groups}

We use the following ordering theorem for abelian groups  to rediscover a single-sided version of the hypersequent calculus for abelian $\ell$-groups defined  in~\cite{MOG05}.

\begin{theorem}[Fuchs 1963~\cite{Fuc63}] \label{t:fuchs}
Every partial order of a torsion-free abelian group $\m{G}$ extends to an order of $\m{G}$. 
\end{theorem}

Let $\vty{A}b$ be the variety of abelian groups and let $\m{T}(k)$ be the algebra of group terms on $k \in \N$ generators. We may identify the free abelian group $\m{F}_{\vty{A}b}(k)$ on $k$ generators with the quotient $\m{T}(k) / \Th_{\vty{A}b}$, where $\Th_{\vty{A}b}$ is the congruence on  $\m{T}(k)$ defined by $s \Th_{\vty{A}b} t \ \Leftrightarrow \ \vty{A}b \models s \eq t$ (see~\cite{BS81} for further details). For convenience, we will use $t \in T(k)$ to denote also $t / \Th_{\vty{A}b}$ in $F_{\vty{A}b}(k)$, noting that $\vty{A}b \models s \eq t$ if and only if $s=t$ in $\m{F}_{\vty{A}b}(k)$.  It follows easily that $\m{F}_{\vty{A}b}(k)$ is torsion-free.

\begin{theorem} \label{t:abelian}
The following are equivalent for $t_1,\ldots,t_n \in T(k)$:
\begin{itemize}

\item[\rm (1)]		$\vty{A} \models \e \le t_1 \lor \ldots \lor t_n$.\smallskip
	
\item[\rm (2)]		$\{t_1,\ldots,t_n\}$ does not extend to an order of $\m{F}_{\vty{A}b}(k)$.\smallskip		

\item[\rm (3)]		$\e \in \sgr{\{t_1,\ldots,t_n\}}$.\smallskip

\item[\rm (4)]		${\vty{A}b} \models \e \eq  t_1^{\la_1}  \cdots t_n^{\la_n}$ for some $\lambda_1,\ldots,\lambda_n \in \N$ not all $0$.

\end{itemize}
\end{theorem}
\begin{proof}
(1) $\Rightarrow$ (2). By contraposition. If $\{t_1,\ldots,t_n\}$ extends to an order of $\m{F}_{\vty{A}b}(k)$, then, taking the inverse order, we  obtain an ordered abelian group where $t_1,\ldots,t_n$ are negative. But this ordered abelian group may also be viewed as an abelian $\ell$-group and taking the  evaluation mapping $t \in T(k)$ to $t \in F_{\vty{A}b}(k)$, we obtain $\vty{A} \not \models \e \le t_1 \lor \ldots \lor t_n$.\smallskip

(2) $\Rightarrow$ (3). Suppose that $\{t_1,\ldots,t_n\}$ does not extend to an order of $\m{F}_{\vty{A}b}(k)$. Then, since $\m{F}_{\vty{A}b}(k)$ is torsion-free, by Theorem~\ref{t:fuchs}, the subsemigroup $\sgr{\{t_1,\ldots,t_n\}}$ is not a partial order of $\m{F}_{\vty{A}b}(k)$. That is, $\e \in \sgr{\{t_1,\ldots,t_n\}}$. \smallskip

(3) $\Rightarrow$ (4). Suppose that $\e \in \sgr{\{t_1,\ldots,t_n\}}$. Then $\e =  t_1^{\la_1}  \cdots t_n^{\la_n}$ in $\m{F}_{\vty{A}b}(k)$ for some $\lambda_1,\ldots,\lambda_n \in \N$ not all~$0$, and hence $\vty{A}b \models \e \eq  t_1^{\la_1}  \cdots t_n^{\la_n}$.\smallskip

(4) $\Rightarrow$ (1). Suppose that ${\vty{A}b} \models \e \eq  t_1^{\la_1}  \cdots t_n^{\la_n}$  for some $\lambda_1,\ldots,\lambda_n \in \N$ not all~$0$. Then also $\vty{A} \models \e \le  t_1^{\la_1}  \cdots t_n^{\la_n}$. It is easily proved that $\vty{A} \models \e \le uv \lor t$ implies $\vty{A} \models \e \le u \lor v  \lor t$ (see, e.g.~\cite{GM16}). Hence, applying this implication repeatedly, we obtain $\vty{A} \models \e \le t_1 \lor \ldots \lor t_n$. 
\qed
\end{proof}

\begin{remark}
Theorem~\ref{t:abelian} may be interpreted geometrically as a variant of Gordan's theorem of the alternative (with integers swapped for real numbers) and close relative of Farkas' lemma (see, e.g.,~\cite{Dan63}). Namely, given an $m \times n$ integer matrix $A = (a_{ij})$, exactly one of the following systems has a solution:
\begin{itemize}
\item[\rm (a)]	$y^T A < 0$  for some $y \in \Z^m$.\smallskip
\item[\rm (b)]	$Az = 0$ for some $z \in \N^n \setminus \{0\}$.
\end{itemize}
To prove this, define $t_i = x_1^{a_{1i}} \cdot \ldots \cdot x_m^{a_{mi}}$ for $i = 1, \ldots, n$. Then (a) is equivalent to $\m{Z} \not \models \e \le t_1 \lor \ldots \lor\, t_n$, which is in turn equivalent to $\vty{A} \not \models \e \le t_1 \lor \ldots \lor \,t_n$  (see Example~\ref{e:abelian}). So, by Theorem~\ref{t:abelian}, (a) fails if and only if $\vty{A}b \models \e \eq  t_1^{\la_1}  \cdots t_n^{\la_n}$ for some $\lambda_1,\ldots,\lambda_n \in \N$ not all $0$, which is in turn equivalent to (b).
\end{remark}

Theorem~\ref{t:abelian} can be used to establish soundness and completeness for the hypersequent calculus $\lgc{GA}$ presented in Figure~\ref{fig:ga}.

\begin{theorem}
For any non-empty hypersequent $\mg$, $\vty{A} \models \mg$ if and only if $\der{GA} \mg$.
\end{theorem}
\begin{proof}
By Theorem~\ref{t:abelian}, $\vty{A} \models \Ga_1 \h \ldots \h \Ga_n$ if and only if $\vty{A}b \models \e \eq \Ga_1^{\la_1}  \cdots \Ga_n^{\la_n}$ for some $\lambda_1,\ldots,\lambda_n \in \N$ not all $0$. But if this latter condition holds, then the number of occurrences of a variable $x$ in $\Ga_1^{\la_1}  \cdots \Ga_n^{\la_n}$ must equal the number of occurrences of $\iv{x}$, and, using $\exr$ and $\idr$, we obtain $\der{GA} \Ga_1^{\la_1}  \cdots \Ga_n^{\la_n}$. Hence also, using $\splitr$ repeatedly, $\der{GA} \Ga_1 \h \ldots \h \Ga_n$. Conversely, we can prove by induction on the height of a derivation that whenever $\der{GA} \Ga_1 \h \ldots \h \Ga_n$, there exist $\lambda_1,\ldots,\lambda_n \in \N$ not all $0$ such that  $\vty{A}b \models \e \eq \Ga_1^{\la_1}  \cdots \Ga_n^{\la_n}$. The cases for $\idr$ and $\exr$ are immediate, and the case of $\splitr$ follows directly by an application of the induction hypothesis.  \qed
\end{proof}

\begin{figure}[t]
 \[
 \begin{array}{ccccccccc}
 \infer[\idr]{\mg \h \De,\ov{\De}}{} & 
 \qquad &
\infer[{\exr}]{\mg \h \PI, \Ga, \De}{\mg \h \PI, \De, \Ga} &
 \qquad & 
 \infer[\splitr]{\mg \h \Ga \h \De}{\mg \h  \Ga, \De} 
\end{array}
\]
\caption{The hypersequent calculus $\lgc{GA}$} \label{fig:ga}
\end{figure}

\begin{remark}
The calculus for abelian $\ell$-groups presented in~\cite{MOG05} uses hypersequents defined as finite multisets of two-sided sequents, each consisting of an ordered pair of finite multisets of $\ell$-group terms, and therefore requires a quite different set of rules. In particular, this calculus contains rules for operation symbols and external contraction and weakening structural rules, but not the exchange rule $\exr$. These differences are of an essentially cosmetic nature, however. We can easily add sound and invertible rules for the operation symbols $\cdot$, $\e$, $\land$, and $\lor$ to the calculus $\lgc{GA}$ that serve to rewrite hypersequents of arbitrary terms into hypersequents built only from literals, and it remains then simply to translate two-sided sequents $\Ga \seq \De$ into one-sided sequents $\ov{\Ga},\De$.  
\end{remark}


\section{Right Orders on Free Groups and Validity in $\ell$-groups}

Let $\vty{G}$ be the variety of groups and $\m{F}(k)$ the free group over $k$ generators, which, as before, we may identify with $\m{T}(k) / \Th_\vty{G}$, where $\Th_\vty{G}$ is the congruence on  $\m{T}(k)$ defined by $s \Th_\vty{G} t \ \Leftrightarrow \ \vty{G} \models s \eq t$. An element of $F(k)$ can again be represented by a term from $T(k)$: in particular, by a reduced term obtained by cancelling all occurrences of $x\iv{x}$ and $\iv{x}x$. Our first aim in this section will be to show that checking validity of equations in $\ell$-groups is  equivalent to checking whether finite subsets of $F(k)$ extend to right orders on $\m{F}(k)$. 

\begin{theorem}\label{t:main}
The following are equivalent for $t_1,\ldots,t_n \in T(k)$:

\begin{itemize}

\item[\rm (1)]		$\vty{LG} \models \e \le t_1 \lor \ldots \lor t_n$.\smallskip

\item[\rm (2)]		$\{t_1,\ldots,t_n\}$ does not extend to a right order of $\m{F}(k)$.\smallskip

\item[\rm (3)]		There exist $s_1,\ldots,s_m \in F(k) \setminus \! \{\e\}$ such that
					\[
					\e \in \sgr{\{t_1,\ldots,t_n,s_1^{\de_1},\ldots,s_m^{\de_m}\}} \
					 \mbox{ for all }\de_1,\ldots,\de_m \in \{-1,1\}.
					\]
				
\end{itemize}

\end{theorem}

Observe that the equivalence of (2) and (3) is an immediate consequence of the following ordering theorem for groups.

\begin{theorem}[Kopytov and Medvedev 1994~\cite{KM94}] \label{t:orderinggroups}
A subset $S$ of a group $\m{G}$ extends to a right order of $\m{G}$ if and only if for all $a_1,\ldots,a_m \in G \setminus\! \{\e\}$, there exist $\de_1,\ldots,\de_m \in \{-1,1\}$ such that $\e \not \in \sgr{S \cup \{a_1^{\de_1},\ldots,a_m^{\de_m}\}}$.
\end{theorem}

\noindent
Condition (3) corresponds directly to derivability in the hypersequent calculus $\lgc{GLG^*}$ presented in Figure~\ref{fig:glstar}.  It is not so easy, however, to show directly that the calculus $\lgc{GLG^*}$ is sound with respect to $\ell$-groups (i.e., to show that $\der{GLG^*} \mg$ implies $\vty{LG} \models \mg$), since the rule $\starr$ is not valid as an implication between premises and conclusion in all $\ell$-groups. We therefore consider also a further hypersequent calculus $\lgc{GLG}$, displayed in Figure~\ref{fig:gl}, and establish the following relationship between the calculi.

\begin{figure}[t]
\[
\begin{array}{ccccc}
\infer[\gvr]{\mg \h \Ga}{} & \qquad & \infer[\splitr]{\mg \h \Ga \h \De}{\mg \h  \Ga, \De} & \qquad & \infer[\starr]{\mg}{\mg \h \De & \mg \h \ov{\De}}\smallskip \\
\mbox{$\Ga$ group valid} & & & & \mbox{$\De$ not group valid.} 
\end{array} 
\]
\caption{The hypersequent calculus $\lgc{GLG}^*$} \label{fig:glstar}
\end{figure}

\begin{figure}[t]
\[
\begin{array}{ccccc}
\infer[\gvr]{\mg \h \Ga}{} & \qquad & \infer[\emr]{\mg \h \De \h \ov{\De}}{} & \qquad &  \infer[{\cutr}]{\mg \h \Ga, \Si}{\mg \h \Ga, \De & \mg \h \ov{\De}, \Si} \smallskip \\
\mbox{$\Ga$ group valid}
\end{array} 
\]
\caption{The hypersequent calculus $\lgc{GLG}$} \label{fig:gl}
\end{figure}

\begin{lemma}\label{l:calculiequivalent}
For any non-empty hypersequent $\mg$, if $\der{GLG^*} \mg$, then $\der{GLG} \mg$.
\end{lemma}
\begin{proof}
It suffices to show that the rules $\splitr$ and $\starr$ of $\lgc{GLG^*}$ are $\lgc{GLG}$-admissible. First, it is easily shown, by an induction on the height of a derivation, that the following rule is $\lgc{GLG}$-admissible:
\[
\infer[\ewr]{\mg \h \mh}{\mg}
\]
 Now for $\splitr$, if $\der{GLG} \mg \h \Ga, \De$, then, by $\ewr$, we obtain $\der{GLG} \mg  \h \Ga, \De \h \De$. But also, by $\emr$, $\der{GLG} \mg  \h \ov{\De} \h \De$, so, by $\cutr$, we obtain $\der{GLG} \mg \h \Ga \h \De$.

To show that $\starr$ is admissible in $\lgc{GLG}$, we consider a restricted version of the calculus where $\cutr$ is never applied to some particular sequent. For a hypersequent $\mg$ and a sequent $\PI$, we call the ordered pair $\langle \PI, \mg \rangle$ a \emph{pointed hypersequent} (just a hypersequent with one sequent marked) and transfer the usual definitions for hypersequent calculi to pointed hypersequent calculi. We let the pointed hypersequent calculus $\lgc{GLG^{p}}$ consist of all pointed hypersequents $\langle \PI, \mg \rangle$ such that either some $\Ga \in \mg \cup \{\PI\}$ is group valid or there exist $\De$ and $\ov{\De}$ in $\mg \cup \{\PI\}$, together with the restricted cut rule
\[
\infer[{\cutr}]{\langle \PI, (\mg \h \Ga, \Si) \rangle}{\langle \PI, (\mg \h \Ga, \De) \rangle & \langle \PI, (\mg \h \ov{\De}, \Si) \rangle}
\]
{\em Claim.}  $\der{GLG} \mg\h \PI$ if and only if $\der{GLG^p} \langle \PI, \mg \rangle$.

\smallskip

\noindent
{\em Proof of Claim.} The right-to-left direction is a simple induction on the height of a derivation of $\langle \PI, \mg \rangle$ in  $\lgc{GLG^{p}}$. For the left-to-right direction, we first note that (by a straightforward induction) whenever $\der{GLG^p} \langle \PI, \mg \rangle$, also $\der{GLG^p} \langle \PI, \mg \hh \mh \rangle$ and $\der{GLG^p} \langle \De, \mg \hh \PI \rangle$. It suffices now to prove that
\[
\der{GLG^{p}} \langle (\Ga, \De), \mg \rangle \ \mbox{ and } \ \der{GLG^{p}} \langle (\ov{\De}, \Si), \mh \rangle 
\quad \Longrightarrow \quad \der{GLG^{p}} \langle (\Ga, \Si), \mg \hh \mh \rangle.
\]
 We proceed by induction on the sum of heights of derivations for $\der{GLG^{p}} \langle (\Ga, \De), \mg \rangle$ and $\der{GLG^{p}} \langle (\ov{\De}, \Si), \mh \rangle$. 
 
 For the base case, there are several possibilities. If $\mg$ or $\mh$ contains a group valid sequent or both $\PI$ and $\ov{\PI}$, then the conclusion follows trivially. If $\Ga, \De$ and $\ov{\De}, \Si$ are both group valid, then $\Ga,\Si$ is group valid and so $\der{GLG^{p}} \langle (\Ga, \Si), \mg \hh \mh \rangle$. Suppose then that $\mg = \mg' \hh \ov{\De}, \ov{\Ga}$, that is, $\der{GLG^{p}} \langle (\Ga, \De), \mg' \hh \ov{\De}, \ov{\Ga} \rangle$. Observe that
\[
\der{GLG^{p}} \langle (\Ga,\Si), \mg' \hh \mh \hh \ov{\De},\Si \rangle
\quad \mbox{and} \quad
\der{GLG^{p}} \langle (\Ga,\Si), \mg' \hh \mh \hh \ov{\Si},\ov{\Ga} \rangle.
\] 
Hence, by $\cutr$, we get $\der{GLG^{p}} \langle (\Ga, \Si), \mg' \hh \mh \hh \ov{\De}, \ov{\Ga} \rangle$; that is, $\der{GLG^{p}} \langle (\Ga, \Si), \mg \hh \mh \rangle$ as required. The case where $\mh = \mh' \hh \ov{\Si}, \De$ is symmetrical.

For the induction step, we apply the induction hypothesis twice to the premises of an application of $\cutr$, and the result follows  by applying $\cutr$.  \qed
\medskip

\noindent
Now to prove that $\starr$ is admissible in $\lgc{GLG^{p}}$, it suffices by the claim to show that for $\De$ not group valid,
\[
\der{{GLG}^{p}} \langle \De, \mg \rangle \ \mbox{ and } \ \der{{GLG}^{p}} \langle \ov{\De}, \mh \rangle 
\quad \Longrightarrow \quad \der{GLG} \mg \h \mh.
\]
We proceed by induction on the height of a ${\lgc{GLG^{p}}}$-derivation of $\langle \De, \mg \rangle$. For the base case, there are several possibilities. If $\mg$ contains a group valid sequent or both $\PI$ and $\ov{\PI}$, then the conclusion follows trivially. Suppose  that $\langle \De, \mg \rangle$ has the form $\langle \De,\mg ' \hh \ov{\De} \rangle$. Since $\der{GLG^{p}} \langle \ov{\De}, \mh \rangle$, also $\der{GLG} \mh \h \ov{\De} \h \mg'$, i.e., $\der{GLG} \mg \h \mh$. For the induction step, suppose that $\mg = \mg' \h \Ga,\Si$ and that $\langle \De, \mg \rangle$ is the conclusion of an application of $\cutr$ with premises $\langle \De,\mg' \hh \Ga,\PI \rangle$ and $\langle \De, \mg' \hh \ov{\PI}, \Si \rangle$. By the induction hypothesis twice, $\der{GLG} \mg' \h \Ga, \PI \h \mh$ and $\der{GLG} \mg' \h \ov{\PI}, \Si \h \mh$. Hence, by $\cutr$, we obtain $\der{GLG} \mg' \h \Ga, \Si \h \mh$; that is, $\der{GLG} \mg \h \mh$. \qed
\end{proof}

\noindent
We now have all the ingredients required to complete the proof of Theorem~\ref{t:main}.

\medskip

\noindent
{\em Proof of Theorem~\ref{t:main}}. 

\smallskip

\noindent
(1) $\Rightarrow$ (2). Suppose contrapositively that $\{t_1,\ldots,t_n\}$  extends to a right order of $\m{F}(k)$. Then the inverse order is a right order $\le$ of $\m{F}(k)$ where $t_1,\ldots,t_n$ are negative. Consider the $\ell$-group ${\m{Aut}}(\langle F(k),\le \rangle)$ and evaluate each variable $x$ by the map $s \mapsto sx$. Then each group term $t$ is evaluated by the map $s \mapsto st$. In particular, each $t_i$ maps $\e$ to $t_i < \e$, and hence $t_1 \lor \ldots \lor t_n$ maps $\e$ to some $t_j < \e$, where $j\in\{1,\dots,n\}$. That is, $\e \not \le t_1 \lor \ldots \lor t_n$ in  ${\m{Aut}}(\langle F(k),\le \rangle)$ and we obtain $\vty{LG} \not \models \e \le t_1 \lor \ldots \lor t_n$.

\smallskip

\noindent
(2) $\Rightarrow$ (3). Immediate from Theorem~\ref{t:orderinggroups}.

\smallskip

\noindent
(3) $\Rightarrow$ (1). Consider $s_1,\ldots,s_m \in F(k) \setminus \! \{\e\}$ where $\e \in \sgr{\{t_1,\ldots,t_n,s_1^{\de_1},\ldots,s_m^{\de_m}\}}$ for all $\de_1,\ldots,\de_m \in \{-1,1\}$. We prove first that $\der{GLG^*} t_1 \h \ldots \h t_n$. For each particular choice of $\de_1,\ldots,\de_m \in \{-1,1\}$, there exist $\lambda_1,\ldots,\lambda_n,\mu_1,\ldots,\mu_m \in \N$ not all $0$ such that $\e = t_1^{\lambda_1} \cdot \ldots \cdot t_n^{\lambda_n} \cdot (s_1^{\delta_1})^{\mu_1} \cdot \ldots \cdot (s_m^{\delta_m})^{\mu_m}$ in $\m{F}(k)$. Hence $\vty{G} \models  \e \eq t_1^{\lambda_1} \cdot \dots \cdot t_n^{\lambda_n} \cdot (s_1^{\delta_1})^{\mu_1} \cdot \dots \cdot (s_m^{\delta_m})^{\mu_m}$ and, by $\gvr$,  $\der{GLG^*} t_1^{\lambda_1} \cdot \dots \cdot t_n^{\lambda_n} \cdot (s_1^{\delta_1})^{\mu_1} \cdot \dots \cdot (s_m^{\delta_m})^{\mu_m}$. But using $\splitr$ repeatedly,  $\der{GLG^*} t_1 \h \ldots \h t_n \h s_1^{\delta_1} \h \ldots \h s_m^{\delta_m}$. So, by applying $\starr$ iteratively, $\der{GLG^*} t_1 \h \ldots \h t_n$. It follows now by Lemma~\ref{l:calculiequivalent} that $\der{GLG} t_1 \h \ldots \h t_n$. But then a simple induction on the height of a derivation in $\lgc{GLG}$, shows that $\vty{LG} \models \e \le t_1 \lor \ldots \lor t_n$ as required. 
\qed

\medskip

\noindent
Soundness and completeness results for $\lgc{GLG^*}$ and $\lgc{GLG}$ follow directly.

\begin{corollary}
The following are equivalent for any hypersequent $\mg$: 
\[
 {\rm (1)}  \ \vty{LG} \models \mg; \quad {\rm (2)} \ \der{GLG} \mg; \quad {\rm (3)} \ \der{GLG^*} \mg.
\]
\end{corollary}


In the last part of this section, we use Theorem~\ref{t:main} to derive new decision procedures for Problems 1 and 2 (see Section~\ref{s:orderedgroups}). Let us denote the {\em length} of a reduced term $t$ in $F(k)$ by $|t|$, and for $N \in \N$, let $F_N(k)$ denote the set of all elements of $\m{F}(k)$ of length $\le N$. Given a subset $S$ of $\m{F}(k)$ which omits $\e$, we call $S$ an {\em $N$-truncated right order} on $\m{F}(k)$ if $S = \sgr{S} \cap F_N(k)$ and, for all $t \in F_{N-1}(k) \setminus \{\e\}$, either $t \in S$ or $\iv{t} \in S$. It has been shown that this notion precisely  characterizes the finite subsets of $F(k)$ that extend to a right order.

\begin{theorem}[Clay and Smith~\cite{Smi05,CS09}]
A finite subset $S$ of $F(k)$ extends to a right order of $\m{F}(k)$ if and only if $S$ extends to an $N$-truncated right order of $\m{F}(k)$ for some $N \in \N$.
\end{theorem}

\noindent 
The condition described in this theorem can be decided as follows. Let $N$ be the maximal length of an element in $S$. Extend $S$ to the finite set $S^{*}$ by adding $st$ whenever $s,t$ occur in the set constructed so far and $|st| \le N$. This ensures that $S^{*} = \sgr{S^{*}} \cap F_N(k)$. If $\e \in S^{*}$, then stop. Otherwise, for every $t\in F_{N-1}(k) \setminus \{\e\}$ such that $t \not \in S^{*}$ and $\iv{t} \not \in S^*$, add $t$ to $S^{*}$ to obtain $S_1$ and $\iv{t}$ to $S^{*}$ to obtain $S_2$, and repeat the process with these sets. This procedure terminates because $F_N(k)$ is finite. Hence we obtain a decision procedure for Problem 1.

\begin{corollary}\label{c:dec}
The problem of checking whether a given finite set of elements of a finitely generated free group extends to a right order is decidable.
\end{corollary}

\noindent
Moreover, using Theorem~\ref{t:main}, we obtain also a decision procedure for Problem 2. 

\begin{corollary}\label{c:hollmc}
The problem of checking whether an equation is valid in all $\ell$-groups is decidable.
\end{corollary}

\begin{example}
Consider $S = \{xx,yy,\iv{x}\iv{y}\} \subseteq F(2)$. By adding all products in $F_2(2)$ of members of $S$, we obtain
\[
S^{*}= \{xx,yy,\iv{x}\iv{y},x\iv{y},\iv{x}y,xy\}.
\]
We then consider all possible signs $\de$ for $x,y \in F_1(2)$. If we add $\iv{x}$ or $\iv{y}$ to $S^{*}$ and take products, then clearly, using $xx$ or $yy$, we obtain $\e$. Similarly, if we add $x$ and $y$ to $S^{*}$, then, taking products, using $\iv{x}\iv{y}$, we obtain $\e$. Hence we may conclude that $S$ does not extend to a right order of $\m{F}(2)$ and obtain
\[
\vty{LG} \models \e \le xx \lor yy \lor \ov{x}\,\ov{y}.
\]
Consider now $T = \{xx, xy,y\iv{x}\} \subseteq F(2)$. By adding all products in $F_2(2)$ of members of $T$, we obtain
\[
T^{*}= \{xx, xy,y\iv{x},yx,yy\}.
\]
We choose $x,y \in F_1(2)$ to be positive and obtain $\{xx, xy,y\iv{x},yx,yy,x,y\}$, a $2$-truncated right order of $\m{F}(2)$. Hence $T$ extends to a right order of $\m{F}(2)$ and
\[
\vty{LG} \not \models \e \le xx \lor xy \lor y\ov{x}.
\]
\end{example}

The decidability result stated in Corollary~\ref{c:hollmc} was first established by Holland and McCleary in~\cite{HM79} using a quite different decision procedure. Let $S$ be a finite set of reduced terms from $\m{F}(k)$. We denote by ${\rm is}(S)$ the set of initial subterms of elements of $S$, and define ${\rm cis}(S)$ to consist of all reduced non-identity terms $\iv{s}t$, where $s,t\in {\rm is}(S)$. The following equivalence (expressed quite differently using ``diagrams'') is proved in~\cite{HM79}.

\begin{theorem}[Holland and McCleary~\cite{HM79}]
The following are equivalent for $t_1,\ldots,t_n \in T(k)$:
\begin{itemize}

\item[\rm (1)]	$\vty{LG} \models \e \le t_1 \lor \ldots \lor t_n$.\smallskip

\item[\rm (2)]	There exist $s_1,\ldots,s_m \in {\rm cis}(\{t_1,\dots,t_n\})$ such that
					\[
					\e \in \sgr{\{t_1,\ldots,t_n,s_1^{\de_1},\ldots,s_m^{\de_m}\}} \
					 \mbox{ for all }\de_1,\ldots,\de_m \in \{-1,1\}.
					\]
\end{itemize}
\end{theorem}

\noindent
Since the set ${\rm cis}(\{t_1,\dots,t_n\})$ is finite and checking $\e \in \sgr{S}$ for a finite subset $S$ of $F(k)$ is decidable, we obtain a decision procedure for Problem~2. Moreover, again using Theorem~\ref{t:main}, we obtain also a decision procedure for Problem~1.

\begin{remark}
Variants of the hypersequent calculi $\lgc{GLG^*}$ and $\lgc{GLG}$ were defined already in~\cite{GM16}, but without the connection to right orders on free groups. They were used to give an alternative proof of Holland's theorem (see~\cite{Hol76}) that the algebra $\m{Aut}(\langle \R,\le \rangle)$  generates the variety $\vty{LG}$ of $\ell$-groups and also to prove that the equational theory of $\ell$-groups is co-NP complete. Let us note here that it follows from the results above that the problem of checking whether a finite subset of $\m{F}(k)$ extends to a right order must also be in co-NP; hardness, however, is still an open problem. Let us also remark that in~\cite{GM16}, the following {\em analytic} (i.e., having the subformula property) hypersequent calculus is shown to be sound and complete for $\ell$-groups:
\[
\begin{array}{ccccc}
\infer[\gvr]{\mg \h \Ga}{} & \quad & 
\infer[{\mixr}]{\mg \h \Ga, \De}{\mg \h \Ga & \mg \h \De} & \quad & 
\infer[{\comr}]{\mg \h \Ga, \De \h \PI, \Si}{\mg \h \Ga, \Si & \mg \h \PI, \De} \smallskip \\
\mbox{$\Ga$ group valid} & & & &
\end{array} 
\]
The proof, however, relies on a rather complicated cut elimination procedure and it is not yet clear how this calculus might relate to right orders on free groups.
\end{remark}


\section{Ordering Free Groups and Validity in Ordered Groups}

In this section, we consider the variety $\vty{RG}$ of {\em representable $\ell$-groups} generated by the class of o-groups. Similarly to the previous section, we establish the following theorem relating validity of equations in this variety (equivalently, the class of o-groups) to extending finite subsets of free groups to (total) orders.

\begin{theorem}\label{t:mainorderedgroups}
The following are equivalent for $t_1,\ldots,t_n \in T(k)$:

\begin{itemize}

\item[\rm (1)]		$\vty{RG} \models \e \le t_1 \lor \ldots \lor t_n$.\smallskip

\item[\rm (2)]		$\{t_1,\ldots,t_n\}$ does not extend to an order of $\m{F}(k)$.\smallskip

\item[\rm (3)]		There exist $s_1,\ldots,s_m \in F(k) \setminus \! \{\e\}$ such that
					\[
					\e \in \nsgr{\{t_1,\ldots,t_n,s_1^{\de_1},\ldots,s_m^{\de_m}\}} \
					 \mbox{ for all }\de_1,\ldots,\de_m \in \{-1,1\}.
					\]
				
\end{itemize}
\end{theorem}

\noindent
In this case, we will not be able to obtain any decision procedure for checking these equivalent conditions. However, we do obtain a new syntactic proof of the orderability of finitely generated free groups~\cite{Neu49}. 

\begin{corollary}
Every finitely generated free group is orderable.
\end{corollary}
\begin{proof}
The equation $\e \le x$ is not valid in the o-group $\m{Z}$, so $\vty{RG} \not \models \e \le x$. But then, by Theorem~\ref{t:mainorderedgroups}, there must exist an order of $\m{F}(k)$ where $x$ is positive. \qed
\end{proof}

\noindent
The proof of Theorem~\ref{t:mainorderedgroups} makes use of the following ordering theorem for groups.

\begin{theorem}[Fuchs 1963~\cite{Fuc63}] \label{t:biorderinggroups}
A subset $S$ of a group $\m{G}$ extends to an order of $\m{G}$ if and only if for all $a_1,\ldots,a_m \in G \setminus\! \{\e\}$, there exist $\de_1,\ldots,\de_m \in \{-1,1\}$ such that $\e \not \in \nsgr{S \cup \{a_1^{\de_1},\ldots,a_m^{\de_m}\}}$.
\end{theorem}

Similarly to the previous section, we introduce hypersequent calculi $\lgc{GRG^*}$ and $\lgc{GRG}$ as extensions of, respectively,  $\lgc{GLG^*}$ and $\lgc{GLG}$ with the rule
\[
\infer[\cycr]{\mg \h \Ga, \De}{\mg \h \De, \Ga}
\]
and establish the following relationship between these calculi.

\begin{lemma}\label{l:rgcalculiequivalent}
For any non-empty hypersequent $\mg$, if $\der{GRG^*} \mg$, then $\der{GRG} \mg$.
\end{lemma}
\begin{proof}
The proof is almost exactly the same as that of Lemma~\ref{l:calculiequivalent} except that we must take account also of the extra rule $\cycr$. That is, we define the pointed hypersequent calculus $\lgc{GRG^p}$ as the extension of $\lgc{GLG^p}$ with the restricted rule 
\[
\infer[\cycr]{\langle \PI, (\mg \hh \Ga, \De) \rangle}{\langle \PI, (\mg \hh \De, \Ga) \rangle}
\]
and prove that $\der{GRG} \mg\h \PI$ if and only if $\der{GRG^p} \langle \PI, \mg \rangle$. In this case, we also prove by   a straightforward induction on the height of a derivation in $\lgc{GRG^{p}}$ that 
\[
\der{GRG^{p}} \langle (\Ga, \De), \mg \rangle \quad \Longrightarrow \quad \der{GRG^{p}} \langle (\De, \Ga), \mg  \rangle.
\]
Finally, the proof that $\starr$ is admissible in $\lgc{GRG^{p}}$ proceeds in exactly the same way as in the proof of  Lemma~\ref{l:calculiequivalent}. \qed
\end{proof}
\medskip

\noindent
{\em Proof of Theorem~\ref{t:mainorderedgroups}}. 

\smallskip

\noindent
(1) $\Rightarrow$ (2).  Suppose contrapositively that $\{t_1,\ldots,t_n\}$  extends to an order of $\m{F}(k)$. Then the inverse order is an order $\le$ of $\m{F}(k)$ where $t_1,\ldots,t_n$ are negative. But this ordered group may also be viewed as a representable $\ell$-group and taking the  evaluation mapping $t \in T(k)$ to $t \in F(k)$, we obtain $\vty{RG} \not \models \e \le t_1 \lor \ldots \lor t_n$.\smallskip

\noindent
(2) $\Rightarrow$ (3). Immediate from Theorem~\ref{t:biorderinggroups}.

\smallskip

\noindent
(3) $\Rightarrow$ (1). Consider $s_1,\ldots,s_m \in F(k) \setminus \! \{\e\}$ where $\e \in \nsgr{\{t_1,\ldots,t_n,s_1^{\de_1},\ldots,s_m^{\de_m}\}}$ for all $\de_1,\ldots,\de_m \in \{-1,1\}$. We prove first that $\der{GRG^*} t_1 \h \ldots \h t_n$. For each  choice of $\de_1,\ldots,\de_m \in \{-1,1\}$, there exist $l > 0$ and conjugates $r_1,\ldots,r_l$ of $t_1,\ldots,t_n,s_1^{\de_1},\ldots,s_m^{\de_m}$ such that $\e = r_1 \cdot \ldots \cdot r_l$ in $\m{F}(k)$. So $\vty{G} \models  \e \eq r_1 \cdot \ldots \cdot r_l$ and, by $\gvr$,  $\der{GRG^*} r_1 \cdot \ldots \cdot r_l$. But then, by $\splitr$ and $\cycr$, also $\der{GRG^*} t_1 \h \ldots \h t_n \h s_1^{\delta_1} \h \ldots \h s_m^{\delta_m}$. Hence, by repeated applications of $\starr$, we get $\der{GRG^*} t_1 \h \ldots \h t_n$. It follows now also by Lemma~\ref{l:rgcalculiequivalent} that $\der{GRG} t_1 \h \ldots \h t_n$. Finally, a simple induction on the height of a derivation in $\lgc{GLG}$ shows that $\vty{RG} \models \e \le t_1 \lor \ldots \lor t_n$ as required. 
\qed

\medskip

\noindent
Soundness and completeness for  $\lgc{GRG^*}$ and $\lgc{GRG}$ follow directly.

\begin{corollary}
The following are equivalent for any hypersequent $\mg$: 
\[
 {\rm (1)}  \ \vty{RG} \models \mg; \quad {\rm (2)} \ \der{GRG} \mg; \quad {\rm (3)} \ \der{GRG^*} \mg.
\]
\end{corollary}


\bibliographystyle{plain}

\end{document}